\documentclass[11pt]{amsart}
\usepackage{amsmath,amssymb, graphicx, amscd,latexsym}
\makeatletter

\usepackage{graphicx, calligra,colortbl}

\newtheorem{Theorem}{Theorem}
\newtheorem{Lemma}[Theorem]{Lemma}
\newtheorem{Corollary}[Theorem]{Corollary}
\newtheorem{Proposition}[Theorem]{Proposition}

\newtheorem{Observation}[Theorem]{Observation}

\newcommand{\eps}{\varepsilon}

\newcommand\al{\alpha}

\newcommand\be{\beta}
\newcommand\Si{\Sigma}

\newcommand\De{\Delta}

\newcommand\cM{\mathcal  M}
\newcommand\cL{\mathcal  L}

\newcommand\BC{ {\mathbb C}}

\newcommand\BZ{{\mathbb  Z}}
\newcommand\BR{ {\mathbb  R}}
\newcommand\BP{ {\mathbb  P}}

\newcommand\bfu{\mbox {\bf  u}}

\newcommand\bfz{\mbox {\bf  z}}

\newcommand\rank{\rm{rank}\/}

\newcommand\rdeg{{\rm{rdeg}\/}}
\newcommand\pdeg{{\rm{pdeg}\/}}

\newcommand\inv{^{-1}}

\def\mapdown#1{\Big\downarrow\rlap{$\vcenter{\hbox{$#1$}}$}}

\def\inv{^{-1}}

\newtheorem{theorem}{Theorem}

\theoremstyle{definition}
\newtheorem{remark}[theorem]{Remark}

\newtheorem{definition}[theorem]{Definition}

\begin{document}
\title[ Topology of  strongly polar weighted homogeneous links
]
{Topology of  strongly polar weighted homogeneous links
}

\author{Vincent Blanl{\oe}il and 
Mutsuo Oka }

\address[V. Blanl{\oe}il]{{IRMA, UFR Math\'ematiques et Informatique\\
Universit\'e de Strasbourg\\
7, rue Ren\'e Descartes\\
F-67084 STRASBOURG}}

\email{v.blanloeil@math.unistra.fr}

\address[M. OKA]{{Department of Mathematics\\
Tokyo  University of Science\\
26 Wakamiya-cho, Shinjuku-ku\\
Tokyo 162-8601}}

\email{\hbox { oka@rs.kagu.tus.ac.jp}}


\begin{abstract}
We consider a canonical $S^1$ action on $S^3$ which is defined by
$(\rho,(z_1,z_2))\mapsto (z_1\rho^p,z_2\rho^q)$ for $\rho\in S^1$ and
 $(z_1,z_2)\in S^3\subset \mathbb C^2$. We consider a link consisting of
 finite orbits of this action, where some of the orbits are reversely oriented.
Such a link appears as a link of a certain type of  mixed polynomials.
We study the space of such links and show smooth degeneration relations.
\end{abstract}
\maketitle

\maketitle

\section{Introduction}
We consider a mixed polynomial
$f(\bfz,\bar \bfz)=\sum_{\nu,\mu}c_{\nu,\mu}\bfz^\nu\bar\bfz^\mu$
where $\bfz=(z_1,\dots,z_n)$, $\bar \bfz=(\bar z_1,\dots,\bar z_n)$,
$ \bfz^\nu=z_1^{\nu_1}\cdots z_n^{\nu_n}$ for $ \nu=(\nu_1,\dots,\nu_n)$
(respectively
$\bar\bfz^\mu=\bar z_1^{\mu_1}\cdots\bar z_n^{\mu_n}$  for 
$ \mu=(\mu_1,\dots,\mu_n)$).

\begin{definition}
We say $f(\bfz,\bar\bfz)$ is {\em a  mixed weighted homogeneous
polynomial of 
radial weight type  $(q_1, \dots, q_n ; d_r) $ and of polar weight type
$(p_1,\dots, p_n;d_p)$} if
\[
\sum_{j=1}^n q_j( \nu_j+\mu_j)=d_r,\quad
\sum_{j=1}^n p_j (\nu_j-\mu_j)=d_p,\quad \text{if}\,\, c_{\nu,\mu}\ne 0.
\]
\end{definition}
Let $f$ be a  mixed weighted homogeneous
polynomial. Using a polar coordinate
$(r,\eta)$ of $\BC^*$ where $r>0$ and $\eta\in S^1$ with
$S^1=\{\eta\in \BC\,|\, |\eta|=1\}$, we define {\em a polar $\BC^*$-action} on
$\BC^n$ by
\begin{eqnarray*}
& (r,\eta)\circ \bfz=(r^{q_1}\eta^{p_1}z_1,\dots, 
r^{q_n}\eta^{p_n}z_n),\quad
(r,\eta)\in \BR^+\times S^1\\
&(r,\eta)\circ\bar \bfz=\overline{(r,\eta)\circ \bfz}=(r^{q_1}\eta^{-p_1}\bar z_1,\dots, 
r^{q_n}\eta^{-p_n}\bar z_n).
\end{eqnarray*}
More precisely, it is a $\mathbb R_+\times S^1$-action.
Then 
$f$ satisfies the
functional equality
\begin{eqnarray}\label{polar-weight}
f((r,\eta)\circ (\bfz,\bar \bfz) )=r^{d_r}\eta^{d_p}f(\bfz,\bar\bfz).\end{eqnarray}
This notion was introduced by  Ruas-Seade-Verjovsky \cite{R-S-V}
and  Cisneros-Molina \cite{Molina}.

A mixed polynomial
$f(\bfz,\bar \bfz)$ is called {\em strongly polar weighted homogeneous}
if  the polar weight and the radial weight coincide, i.e.,  $p_j=q_j,\,1\le j\le n$.

In this case, the $\mathbb C^*$ action is simply defined by
\[
 \zeta\circ \bfz=(z_1\zeta^{p_1},\dots, z_n\zeta^{p_n}),\quad \zeta\in
 \mathbb C^*.
\]
In this paper, we study the geometry of the links defined by strongly
polar weighted homogeneous
mixed polynomials.

\section{Cobordism of links}

First of all we have to point out that the topology of mixed links
is very particular and we recall some classical results and definitions 
in the case of knots and algebraic links. 

Let $K$ be a closed $2k-1$-dimensional manifold embedded
in the $(2k+1)$-dimensional sphere $S^{2k+1}$.
We suppose that $K$ is 
 $(k-2)$-connected if $\geq 2$.
When $K$ is orientable, we further assume that 
it is oriented. Then we call $K$ or its (oriented) isotopy class
an \emph{$2k-1$-knot}

\begin{definition}\label{dfn:cob}
Two $2k-1$-knots $K_0$ and $K_1$ in $S^{2k+1}$ are 
said to be \emph{cobordant}\index{knot!cobordant} if there exists a 
properly embedded\footnote{Recall that a manifold with boundary
$Y$ embedded in a manifold $X$ with boundary
is said to be \emph{properly embedded} if
$\partial Y = \partial X \cap Y$ and $Y$
is transverse to $\partial X$.}
$(2k)$-dimensional
manifold $X$ of \linebreak $S^{2k+1} \times [0,1]$ such that
\begin{enumerate}
\item $X$ is diffeomorphic to $K_0 \times [0,1]$, and
\item $\partial X = (K_0 \times \{0\}) \cup (K_1 \times \{1\}).$
\end{enumerate}
The manifold $X$ is called a \emph{cobordism}\index{cobordism}
between $K_0$ and $K_1$.
When the knots are oriented, we say that 
$K_0$ and $K_1$ are \emph{oriented cobordant} 
(or simply \emph{cobordant}) if 
there exists an oriented cobordism $X$
between them such that $$\partial X = 
(-K_0 \times \{0\}) \cup (K_1 \times \{1\}),$$
where $-K_0$ is obtained from $K_0$ by reversing
the orientation.
\end{definition}

\begin{figure}[tb]
\centering
 \begin{picture}(100,120)(0,0)
 \thicklines
 \qbezier(10,90)(30,50)(15,10)
 \put(21,57){\circle*{3}}
  \put(22,40){\circle*{3}}
 \put(8,48){\makebox{\small$K_{0}$}}
  \put(-5,100){\makebox{\small $ S^{n+2}\times\{0\}$}}
\thinlines
 \qbezier[30](15,90)(35,50)(20,10)
 \qbezier[30](20,90)(40,50)(25,10)
 \qbezier[30](25,90)(45,50)(30,10)
 \qbezier[30](30,90)(50,50)(35,10)
 \qbezier[30](35,90)(55,50)(40,10)
 \qbezier[30](40,90)(60,50)(45,10)
 \qbezier[30](45,90)(65,50)(50,10)
 \qbezier[30](50,90)(70,50)(55,10)
 \qbezier[30](55,90)(75,50)(60,10)
 \qbezier[30](60,90)(80,50)(65,10)
 \qbezier[30](65,90)(85,50)(70,10)
 \qbezier[30](70,90)(90,50)(75,10)
 \qbezier[30](75,90)(95,50)(80,10)
 \qbezier[30](80,90)(100,50)(85,10)
 \qbezier[30](85,90)(105,50)(90,10)
\color{red} \qbezier(21,57)(95,35)(65,55)
\qbezier(65,55)(35,75)(101,57)

 \qbezier(22,40)(35,20)(102,40)\color{black}
\thicklines
 \qbezier(90,90)(110,50)(95,10)
 \put(101,57){\circle*{3}}
 \put(102,40){\circle*{3}}
 \put(104,48){\makebox{\small$K_{1}$}}
  \put(85,100){\makebox{\small $ S^{n+2}\times\{1\}$}}
   \put(35,0){\makebox{\small $ S^{n+2}\times[0, 1]$}}
 \end{picture}
\caption{A cobordism between $K_0$ and $K_1$}
\label{fig2cob}
\end{figure}
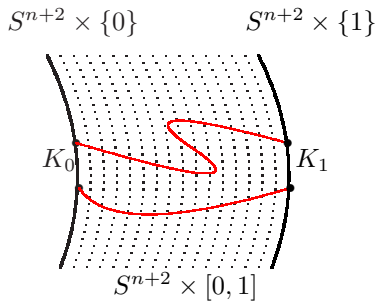

It is clear that isotopic knots are always cobordant. However,
the converse is not true in general (see Fig.~\ref{figcobiso}).

For a classification of high dimensional knots up to cobordism we refer to \cite{VB}.

\begin{figure}[tb]
\centering
 \begin{picture}(100,100)(0,0)
 \put(10,50){\circle*{3}}
 \put(9,53){\makebox{$K_{0}$}}
 \qbezier(10,50)(20,45)(29,49)
 \qbezier(32,50 )(90,80)(78,53)
 \qbezier(77,50)(52,0)(30,50)
 \qbezier(30,50)(15,90)(54,62)
\qbezier(56,60)(75,50)(90,50)
 \put(90,50){\circle*{3}}
 \put(89,53){\makebox{$K_{1}$}}
 \end{picture}
\caption{A cobordism which is not an isotopy}
\label{figcobiso}
\end{figure}
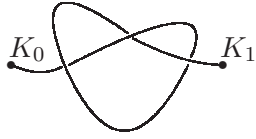

Let us study one example of dimensional one links. We denote by $T^{+}$ and $T^{-}$ respectively the one dimensional right and
the left trefoil knots (which are both mixed links). We know that $T^{+}$ and $T^{-}$ are cobordant, see
\cite{DR} p. 219 ; but let us give here the idea of the proof.

Precisely, 
we denote by $S_{+}^{3}$ (resp. $S_{-}^{3}$) the upper (resp. lower) hemisphere  of the unit $3$-sphere $\partial D^{4} = S^{3}\hookrightarrow \mathbf{R}^{4}$. 
Set $\mathcal{E}$ be the equatorial hyperplane  of $D^{4}$, and 
let $\pi: \mathbf{R}^{4} \rightarrow \mathcal{E}$ the orthogonal projection onto $\mathcal{E}$.

One can suppose that $T^{+}$ and $T^{-}$, which is the mirror image of $T^{+}$, are embedded in $S_{+}^{3}$ and $S_{-}^{3}$ respectively such that\footnote{The sign is necessary to have the right orientation.}
$$T^{-} = -
\bigl(\pi(T^{+}) \times [0, 1]\bigr) \cap S_{-}^{3}.$$ 

Then we construct 
the connected sum $\mathcal{O}=T^{+}\# T^{-}$ of $T^{+}$ and $T^{-}$ in $S^{3}$; we illustrate this construction in Fig.~\ref{fig3.1}.

Set ${\tilde{T}^{+}}$ (resp. ${\tilde{T}^{-}}$) the intersection ${\tilde{T}^{+}} = \mathcal{O} \cap{S^{3}_{+}}$ 
(resp. ${\tilde{T}^{}} = \mathcal{O} \cap{S^{3}_{-}}$).  One can assume that the connected sum $\mathcal{O}$ is made in order to have $${\tilde{T}^{-}} = -
\bigl(\pi({\tilde{T}^{+}}) \times [0, 1]\bigr) \cap S_{-}^{3}.$$

Now, if we denote $${\mathcal{D}} = \bigl(\pi({\tilde{T}^{+}})\times [0,1]\bigr)\cap D^{4},$$ then 
${\mathcal{D}}$ is homeomorphic to a $2$-disk  since $\pi({\tilde{T}^{+}})$ is a $1$-disk. Moreover $\partial \mathcal{D} = \mathcal{O} = T^{+}\# T^{-}$.
Since $\mathcal{O}$ bounds a $2$-disk embedded in $D^{4}$ then $\mathcal{O }$ is null cobordant, and, $T^{+}$ and $T^{-}$ are cobordant.\footnote{Remark that 
the trefoil knot is homeomorphic to a sphere, then to prove that $T^{+}$ and $T^{-}$ are cobordant it is sufficient to prove that their connected sum bounds a disk \cite{K}.}

\begin{figure}[tb]
\centering
 \begin{picture}(120,120)(0,0)
\thinlines\color{red}
 \qbezier(16,25 )(20,28)(26,30)
 \qbezier(28,32 )(45,40)(39,26.5)
 \qbezier(39,26.5)(26,0)(15,23)
\qbezier(14,25)(8.5,35)(12,35)
\qbezier(16,35)(22,34)(27,31)
\put(30,76.5){\vector(1,0){0}} \put (50, 77){$T^{+}$}

\qbezier(12,35)(10,39)(10,42)
\qbezier(16,35)(14,39)(14,41.5)

\qbezier(40,25)(47.5,21)(45,15)
\qbezier(37.5,26)(31,30)(27,31)
\qbezier(45,15)(44,7.5)(25,7.5)
\qbezier(25,7.5)(-5,8.5)(16,25)

\qbezier(16,58)(45,43)(39,55.5)
 \qbezier(38,58)(26,83)(15,58)
     \qbezier(15,58)(8.5,48)(12,48)
\qbezier(16,48)(22,49)(26.5,52)
\put(24,7.5){\vector(-1,0){0}} \put (50, 8){$T^{-}$}
\qbezier(12,48)(10,44)(10,42)
\qbezier(16,48)(14,44)(14,41.5)

\qbezier(29,53.5)(47.5,58)(45,68)
\qbezier(45,68)(44,76.5)(25,76.5)
\qbezier(25,76.5)(-5,74.5)(14.5,59)

\thicklines
\color{blue}\put(114, 37){$\mathcal{E}$}
\qbezier(-10,50)(-10,40)(50,40)
\qbezier(50,40)(110,40)(110,50)

\qbezier[40](-10,50)(-10,63)(50,63)
\qbezier[40](50,63)(110,63)(110,50)
\color{black}
\qbezier(-10,50)(-10,110)(50,110)
\qbezier(50,110)(110,110)(110,50)
\qbezier(110,50)(110,-10)(50,-10)
\qbezier(50,-10)(-10,-10)(-10,50)

\put(115,80){$S^3$}

 \end{picture}
\caption{The connected sum of the trefoil knot and its inverse in $S^{3}$ }
\label{fig3.1}
\end{figure}
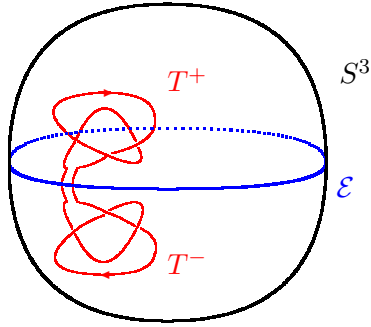

\medskip

In \cite{Le}~D.~T.~L\^e proved that the Alexander
polynomial determines the topological type of the link of an isolated singularity
of a complex analytic curve and moreover he proved that cobordant links are isotopic
since the product of their Alexander polynomials is a square.

In the case of mixed links things are different. For example the two trefoil knots $T^{+}$ 
and $T^{-}$ are cobordant but not isotopic mixed links. Recall that they are not
isotopic since they have distincts Jones polynomials.

\begin{remark}
Moreover, since the trivial knot $\mathcal{O}$ is a mixed link, then the connected sum of mixed one dimensional links can be a mixed link
contrary to the classical case as proved by N.~A'Campo~\cite{ACampo}.
\end{remark}

 \section{Strongly polar weighted homogeneous links}
 Hereafter we consider strongly polar weighted homogeneous polynomial
 $f(\bfz,\bar\bfz)$ of two variables i.e., $n=2$
with weight vector $P={}^t(p,q)$. Here we assume that $\gcd(p,q)=1$.
We assume that $f$ is convenient and non-degenerate so that 
 the link $L=f\inv(0)\cap S^3$ is smooth.
 Let $\cM(P;d_p)$ be the space of strongly polar weighted homogeneous mixed
 polynomials of  the polar degree $d_p$,
 which is non-degenerate convenient and let $\cL(P;d_p)$
 be the associated oriented links. Hereafter we denote simply $\cM, \cL$ for  $\cM(P;d_p)$ 
and $\cL(P;d_p)$ respectively.
We have a canonical mapping
$\pi:\cM\to \cL$ defined by $\pi(f)$ the link defined by 
$f\inv(0)\cap S^3$. A difficulty in the mixed polynomial situation is
 that for a fixed link, there  exist an infinitely many   mixed polynomials which define
 the link.

 Let $d_r,\,d_p$ be the radial and polar degrees respectively.
As $f$ is assumed to be convenient,
 $f$ contains  monomials
 $z_1^{a_1}\bar z_1^{b_1}$ and $z_2^{a_2}\bar z_2^{b_2}$ such that 
 \[
 p(a_1+b_1)=q(a_2+b_2)=d_r,\, \quad p(a_1-b_1)=q(a_2-b_2)=d_p.
 \]
 Therefore
 \[
 \frac{p}{q}=\frac{a_2}{a_1}=\frac{b_2}{b_1} 
 \]
and we see that $p|a_2,b_2$ and $q|a_1,b_1$ and thus $pq|d_r,d_p$.
 As our link is $S^1$ invariant,
its component is a finite union of orbits of the action.
 Recall that 
 the associated $S^1$-action is defined by 
 \[
 S^1\times S^3\to S^3,\, (\rho,(z_1,z_2))\mapsto (z_1\rho^p,z_2\rho^q),\, \rho\in S^1
 \]
 Let $P=(p,q)$ be the primitive weight vector of $f$. $P$ is fixed throughout this paper.
Note that  $L$ is stable under the action, by the Euler equiality
\[
 f(\rho\circ \bfz)=\rho^{d_p}f(\bfz).
\]
 Two orbit $z_1=0$ and $z_2=0$ are singular but by the covenience
 assumption, our link has only regular orbits.
\subsection{Coordinates of the orbits}
Take a regular orbit $L$. We can take a point
$X=(\be_1,\be_2)\in L\subset S^3\subset \BC^2$ such that $\be_1$ is a positive  number. $\be_1$ and $|\be_2|$
are unique by $L$ but $\be_2$ is not unique. The umbiguity is the action 
of $\BZ/p\BZ$. Thus $|\be_2|=\sqrt{1-\be_1^2}$ and the argument of 
$\be_2$ is unique mudulo $2\pi/p$.
Thus the space of the regular orbits is isomorphic to the 
punctured disk $\Delta^*:=\{\xi=r\rho\in \BC\,|\, 0<r<1,\,\rho\in S^1\}$,
by the correspondence $\beta_2\mapsto \beta_2^p\in \De^*$. 
For
$u=r^p\,e^{i\theta}\in \Delta_p^*$, we associate the regular orbit
\begin{eqnarray}
K(u):=\{ (\rho^p\sqrt{1-r^2},\rho^q re^{i\theta/p})\,|\,\rho\in S^1\},\quad u=re^{i\theta}\in \Delta^{*}.
\end{eqnarray}
Consider a strongly polar weighted homogeneous polynomial for arbitrary
non-negative integer $k$:
\begin{eqnarray}
&\begin{cases}
\ell_{u,k}(\bfz)&:=z_1^{q+kq}\bar z_1^{kq} - \al_{u,k} z_2^{p+kp}\bar z_2^{kp}=z_1^{q}\|z_1^{q}\|^{2k} - \al_{u,k} z_2^{p}\| z_2^{p}\|^{2k}\\
{\bar\ell}_{u,k}(\bfz)&:=\bar z_1^{q}\|z_1^{q}\|^{2k} - \overline{\al}_{u,k} \bar z_2^{p}\| z_2^{p}\|^{2k}\quad\text{where}
\end{cases}\,\\
&\al_{u,k}=\frac{(1-r^2)^{q(1/2+k)}} {r^{p(1+2k)} e^{i\theta}}.
\end{eqnarray}
Note that the polar degrees of $\al_{u,k}$  are $pq$  but the radial
degrees
are different and they are given as

\[
 \rdeg\,\ell_{u,k}=(2k+1)pq.
\]
\begin{Observation}
The polynomials $\ell_{u,k}$ define $K(u)$
and $\bar \ell_{u,k}$ defines $K(u)$ with  reversed orientation for any $k=0,1,\dots$
\end{Observation}
Hereafter we simply use the notation:
\[
 \ell_u(\bfz):=\ell_{u,0}(\bfz)=z_1^q-\frac{(1-r^2)^{q/2}}{r^p e^{i\theta}}z_2^p,\,\,u=r^p\,e^{i\theta}.
\]
Let $\cL(P;dpq,r)$ be the subspace of $L(P;dpq)$ which has $d+2r$ components
where $r$ components are negatively oriented.
First we prepare the next lemma:
\begin{Lemma}
The moduli space $\cL(P;dpq,r)$ is connected and therefore any two links
 of this moduli has the same topology.
\end{Lemma}
\begin{proof} Note that $\cL(P;dpq,r)$ are 
parametrized by 
$$M_{d,r}:=(\Delta^*)^{d+2r}\setminus \Xi$$
where $\Xi=\{\bfu=(u_1,\dots, u_{d+2r})\in {\Delta^*}^{(d+2r)}\,|\, u_i=
 u_j\,(\exists i,j,\,i\ne j) \}$.
 Thus it is easy to see that $M_{d+2r}$ is connected.
$\bfu$ corresponds to the link $\cup_{i=1}^{d+2r}K(u_i)$ where 
$K(u_j)$ are reversely oriented for $j=d+r+1,\dots, d+2r$.
\end{proof}
\subsection{Typical degeneration}
We consider an important degeneration of links
$L(t),\,t\in \BC$ which is defined by the family of strongly polar weighted homogeneous polynomials:
\[
f(\bfz,\bar\bfz,t)=-2 z_2^{2p}{\bar z_2}^p+z_1^{2q}{\bar z_1}^q+t z_2^{2p}{\bar z_1^q}.
\]
Using Proposition 1 (\cite{OkaPolar}), we see that 
the degeneration locus is given as the following real semi-algebraic variety
\[
 \Si:=\{t\in \mathbb C\,|\, t= \frac{2 s-1}{s^2},\,\exists s\in S^1\}
\]
Figure 1  shows the graph
of $\Si$. Let $\Omega$ be the bounded region surrounded by $\Si$.
By
Example 59 in \cite{OkaMix}, we can see the following.
{\begin{Proposition}
For any $t\in \Omega$,
 $L(t)$ has  one link component,
  while for $t\in \mathbb C\setminus\bar\Omega$(= the outside of $\Si$),
$L(t)$ has three
components. 
\end{Proposition}
\begin{proof}Let us consider the weighted projective space $\BP^1(P):=\BC^2\setminus\{O\}/\mathbb C^*$
by the above $\BC^*$-action. For $U:=\BP^1(P)\cap \{z_1z_2\ne 0\}$, it is easy to see that 
$u:=z_2^p/z_1^q$ is a coordinate function.
Our link corresponds to the solutions (=zero points) of 
\[
 -2 u^{2}\bar u+t u^{2}+1=0
\]
and there exists one solution  (respectively 3 solutions) for each $t\in \Omega$ (resp.
 $t\notin {\bar\Omega}$).
See Example 59, \cite{OkaMix} or \cite{MixIntersection}.
\end{proof}
We consider the point $-3\in \Si$ which is a smoot point of $\Si$.
 There are two components for $L(-3)$ ($u=1/2$ and $u=-1$) and the component passing 
through
$(1,e^{i\pi /p})$ is a doubled component. 
Here we are considering the link on the sphere of radius $\sqrt 2$,
$S_{\sqrt 2}$ for simplicity.
Let us consider the variety:
\[
\mathcal W=\{(z_1,z_2,t)\in S_{\sqrt 2}\times \mathbb R\,|\, 
-3-\eps\le t\le-3+\eps,\,
f(\bfz,\bar\bfz,t)=0\},\quad \eps\ll 1.
\]
The following is the key assertion. 
\begin{Lemma}
$\mathcal W$ is a smooth manifold 
with boundary $L(-3-\eps) \cup -L(-3+\eps)$.
\end{Lemma}
\begin{proof}
Let $f(\bfz,\bar\bfz,t)=g(\bfz,\bar \bfz,t)+i\,h(\bfz,\bar\bfz,t)$.
We assert that $\mathcal W$ is a complete intersection variety. For this purpose, we show that 
three 1-forms $dg,dh, d\rho$ are independent on $L(-3)$,
where 
$\rho(\bfz)=\|\bfz\|^2$. 
As the polynomial $f$ is strongly polar weighted homogeneous, it is enough to check the assertion
on a point $\tilde{\bfz}_0=(1,\al,-3)\in W$ where $\al=e^{i\pi /p}$.
For the calculation's simplicity, we use the base $\{dz_1,d\bar z_1,dz_2,d\bar z_2,dt\}$ of the complexified cotangent space.
Using the equalities
$g=(f+\bar f)/2,\,h=(f-\bar f)/(2i)$, we get
\[\begin{split}
&\left(\begin{matrix}dg(\tilde{\bfz}_0)\\dh(\tilde{\bfz}_0)\\d\rho(\tilde{\bfz}_0)\end{matrix}\right)=A
\left(\begin{matrix}
dz_1\\d\bar z_1\\dz_2\\d\bar z_2\\dt
\end{matrix}
\right) \qquad 
\text{where}\\
&A= \left[ \begin {array}{ccccc} 0&0&0&0&1\\ \noalign{\medskip}-2\,iq&2\,
iq&2\,ip{\it \bar\al}&-2\,ip{\it \al}&0\\ \noalign{\medskip}1&1&{\it \bar \al}&{
\it \al}&0\end {array} \right] 
\end{split}
\]
Thus it is easy to see that $\rank\, A=3$.
\end{proof}
\vspace{1cm}
\begin{figure}
{\includegraphics[width=6cm,height=5cm]{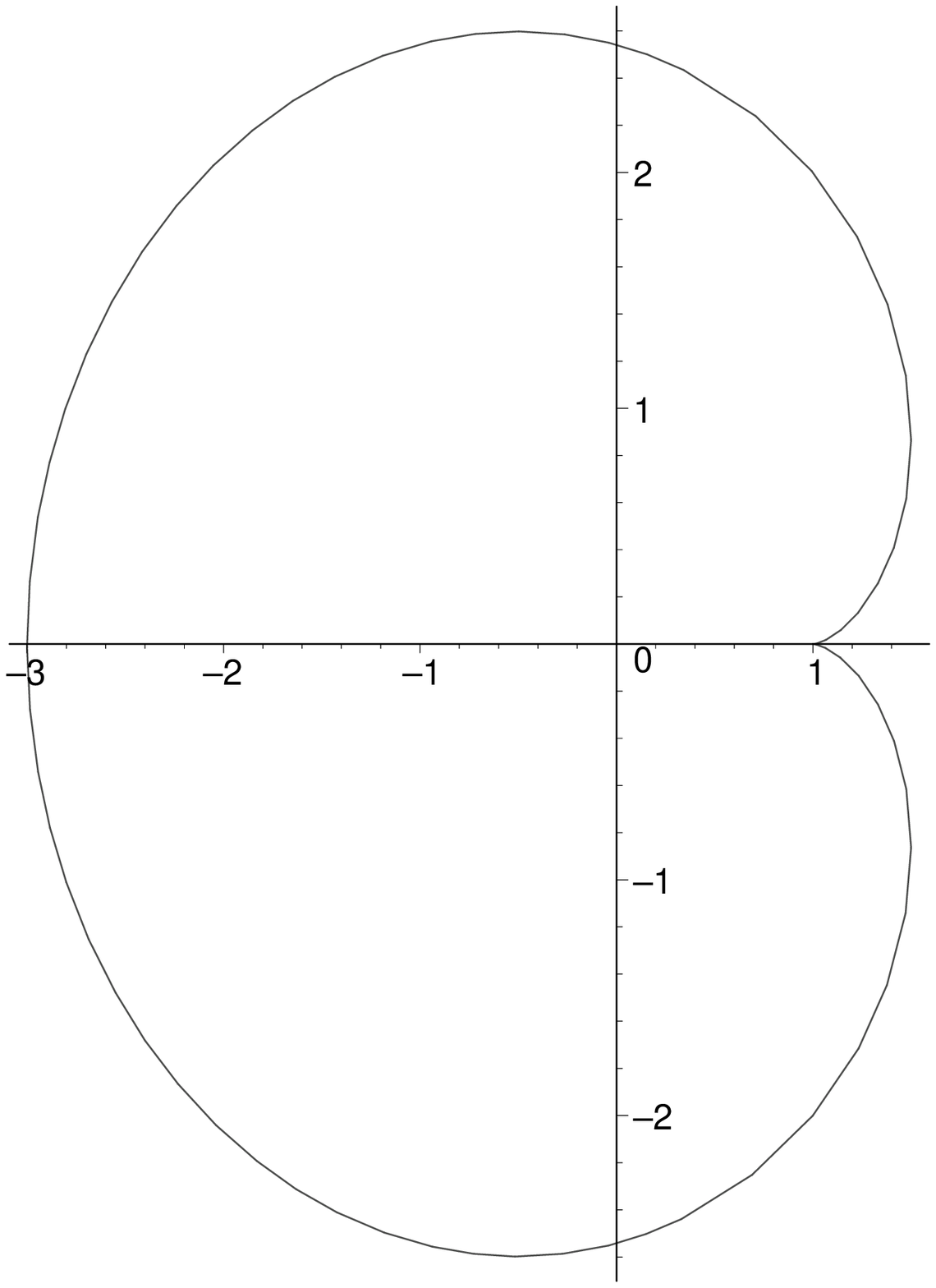}}
\caption{\label{sigma}$\Sigma$}
\end{figure}
\subsection{Milnor fibrations}
Take $\bfu=(u_1,\dots, u_{d+2r})\in M_{d,r}$ and consider 
the corresponding link $L(\bfu)=\cup_{j=1}^{d+2r}K(u_j)$ with 
$d+2r$ components and the last $r$ components are negatively oriented.
Let $f(\bfz)$ be a strongly polar weighted homogeneous  polynomial which defines
$L(\bfu)$
with $\pdeg\, f=dpq$ and $\rdeg\,f= (d+2s)pq$ with $s\ge r$.
For example, we can take
\[
 g(\bfz)=\ell_{u_1,s-r}(\bfz)\prod_{j=2}^{d+r}\ell_{u_j}(\bfz)\prod_{j=d+r+1}^{d+2r}\bar\ell_{u_j}(\bfz).
\]
Let $F$ be the Milnor fiber of $f$:
$F=\{\bfz\in S^3\,|\, f(\bfz)>0\}$.
As we assume that  $L(\bfu)$ has no singular orbit, $f(\bfz)$ is a convenient mixed
polynomial.
Thus it contains monomials
$z_1^{(d+s)q}\bar z_1^{qs}$ and $z_2^{(d+s)p}\bar z_2^{ps}$.
The monodromy $h:F\to F$ is defined by
$h(\bfz)=e^{2\pi i/dpq}\circ \bfz$ and it is the restriction of
$S^1$-action
to $\mathbb Z_{dpq}\subset S^1$.
Thus we have a commutative diagram:
\[
 \begin{matrix}
  F&\hookrightarrow& S^3\setminus L(\bfu)\\
 &\searrow\,\xi&\mapdown{\pi}\\
&& \mathbb P^1(P)\setminus W
 \end{matrix}
\]
where $W$ is $d+2r$ points corresponding to the components of $L(\bfu)$.
 $\pi,\xi$ are canonical quotient mapping by $S^1$ and
$\mathbb Z_{dpq}$ respectively.
As $F$ is a $\mathbb Z_{dpq}$ cyclic covering over $\mathbb P^1\setminus W$, with two
 singular points $(0,1)$ and $(1,0)$. Over these two points, the
 corresponding fibers are $q,p$ points respectively. Thus we have
\begin{Proposition}(cf. Theorem 65,\cite{OkaMix})
 The Euler charactersitic of $F$ is given as 
\[
 \chi(F)=-(d+2r)dpq+p+q
\]
\end{Proposition}
Note that $\chi(F)$ depends on the number of components $d+2r$ but
it does  not depend
on
the radial degree $(d+2s)pq$.
Thus we see that, under  fixed polar and radial degrees,
there are $s+1$ different topologies among their Milnor fibrations.
The components types can be $d+2r,\,r=0,\dots, s$.

\section{main result}
Consider a smooth family of strongly polar  weighted homogeneous  links
$L(t)\in \mathcal L(P;dpq), 0\le t\le 1$
with weight $P={}^t(p,q)$
such that 

(1) the variety $W=\{(\bfz,t)\in S^3\times [0,1]\, \bfz\in L(t)\}$ is a smooth variety of codimension two.

(2) There exists $t_0,\,0<t_0<1$ such that 
\[
\begin{cases} L(t)\in 
 \cL(P;dpq,r-1)\quad &t< t_0\\
L(t)\in \cL(P;dpq,r)\quad &t>t_0
\end{cases}
\]
The link $L(t_0)$ is singular. One component is the limit of two
components with  opposite orientations.
We call such a family {\it a smooth elimination of a pair of links}.

\begin{Theorem}For any link $L\in \cL(P;dpq,r)$ with $r>0$,
there exists a smooth elimination family $L(t)$ of a pair of links with $L(0)=L$
and $L(1)\in \cL(P;dpq,r-1)$.
\end{Theorem}

\begin{Corollary}For any link $L\in \cL(P;dpq,r)$ with $r>0$,
$r$ pairs of links with opposite orientations
can be eliminated successively to a link $L'\in \cL(P;dpq,0)$ of positive
 link.
$L'$ is isomorphic to  a holomorphic torus link defined by
\[
 z_1^{qd}-z_2^{pd}=0.
\]
\end{Corollary}

\def\cprime{$'$} \def\cprime{$'$} \def\cprime{$'$} \def\cprime{$'$}
  \def\cprime{$'$}


\end{document}